\documentclass[11pt,a4paper]{article}
\usepackage{amssymb,amsmath,amsthm}

\theoremstyle{plain}
\newtheorem{theorem}{Theorem}[section]
\newtheorem{lemma}[theorem]{Lemma}
\newtheorem{corollary}[theorem]{Corollary}
\newtheorem{conjecture}[theorem]{Conjecture}

\theoremstyle{definition}

\newtheorem{defn}[theorem]{Definition}

\newtheorem{claim}[theorem]{Claim}

\newtheorem*{remark}{Remark}

\newcommand\lref[1]{Lemma~\ref{lem:#1}}
\newcommand\tref[1]{Theorem~\ref{thm:#1}}
\newcommand\cref[1]{Corollary~\ref{cor:#1}}

\newcommand\cjref[1]{Conjecture~\ref{conj:#1}}

\begin{document}

\title{A discrete isodiametric result: the Erd\H{o}s-Ko-Rado theorem for multisets}

\author{Zolt\'an F\"uredi \thanks{Alfr\'ed R\'enyi Institute of Mathematics,
13--15 Re\'altanoda Street, 1053 Budapest, Hungary. 
email: z-furedi@illinois.edu.
Research supported in part by the
Hungarian National Science Foundation OTKA 104343,
and by the European Research
Council Advanced Investigators Grant 267195.} \and
D\'aniel Gerbner\thanks{Alfr\'ed R\'enyi Institute of Mathematics,
13--15 Re\'altanoda Street, 1053 Budapest, Hungary. \newline
E-mail: gerbner.daniel@renyi.mta.hu.
Research supported by the Hungarian National Scientific Fund, 
grant number: PD-109537.} \and
M\'at\'e Vizer\thanks{Alfr\'ed R\'enyi Institute of Mathematics,
13--15 Re\'altanoda Street, 1053 Budapest, Hungary. \newline
E-mail: vizermate@gmail.com
Research supported by the Hungarian National Scientific Fund, 
grant number: 83726.}}

\maketitle

\begin{abstract}
There are many generalizations of the Erd\H{o}s-Ko-Rado theorem.
We give new results (and problems) concerning families of $t$-intersecting 
$k$-element
multisets of an $n$-set and point out connections to coding theory and 
classical geometry.
We establish the conjecture that for $n \geq  t(k-t)+2$
such a family can have at most ${n+k-t-1\choose k-t}$ members.
\end{abstract}

\section{Introduction}

\subsection{The isodiametric problem}

In 1963 Mel'nikov~\cite{Me} proved that the ball has the maximal volume among all sets with a given diameter in every Banach space (of finite dimension). 
We call the problem of finding the maximal volume among the sets with given diameter in a metric space the \textit{isodiametric problem}. Various results have been achieved concerning the discrete versions of this problem.

Kleitman~\cite{Kl} as a slight generalization of a theorem of Katona~\cite{Ka} determined the maximal volume among subsets with diameter of $r$ in $\{0,1\}^n$ with the Hamming distance (that is, the distance of $(x_1,..,x_n), (y_1,...,y_n) \in \{0,1\}^n$ is $|\{i \leq n : x_i \ne y_i \}|$) and proved that it is achieved if the subsets is a ball of radius $r/2$ if $r$ is even. Ahlswede and Khachatrian~\cite{AK2} generalized this result to $[q]^n$ and solved the isodiametric problem for all $q,n$ and diameter $r$.

Du and Kleitman~\cite{DK} considered and Bollob\'as and Leader~\cite{BL} completely solved the isodiametric problem in $[k]^n$ with the $\ell_1$ distance (that is, the distance of $(x_1,..,x_n), (y_1,...,y_n) \in [q]^n$ is $\sum_{i=1}^n|x_i-y_i|$).

\subsection{Erd\H{o}s-Ko-Rado type theorems}

Let us call a set system ${\mathcal F}$ \textit{intersecting} if $|F_1 \cap F_2| \ge 1$ for all $F_1,F_2 \in {\mathcal F}$. It is easy to see that the cardinality of an intersecting set system of subsets of $[n]$ is at most $2^{n-1}$. However if we make restrictions on the cardinality of the subsets, the problem becomes more difficult.

Let us use the following notation $\binom{[n]}{k}:=\{A \subseteq [n] : |A|=k\}$. In the 1930's Erd\H{o}s, Ko and Rado proved (and published in 1961) the following theorem:

\begin{theorem}(\cite{EKR})
\label{thm:EKRorig}
If $n \geq 2k$ and ${\mathcal F} \subseteq \binom{[n]}{k}$ intersecting then

$$|{\mathcal F}|\leq \binom{n-1}{k-1}.$$
\end{theorem}

Observe that if one considers the indicator functions of subsets of $[n]$ as elements of $\{0,1\}^n$ (with Hamming or $\ell_1$ distance is the same in this case), then the intersecting property of ${\mathcal F} \subseteq \binom{[n]}{k}$ is equivalent with the fact that the diameter of the set of the indicator functions of the elements of ${\mathcal F}$ is at most $2k-1$. So as the inequality is sharp in the Erd\H{o}s-Ko-Rado theorem, it solves an isodiametric problem.

\vspace{3mm}

Let us call a set system ${\mathcal F}$ \textit{t-intersecting} if $|F_1 \cap F_2| \ge t$ for all $F_1,F_2 \in {\mathcal F}$. Erd\H{o}s, Ko and Rado also proved in the same article that if $n$ is large enough, every member of the largest $t$-intersecting family of $k$-subsets of $[n]$ contains a fixed $t$-element set, but did not give the optimal threshold. Frankl~\cite{Fr} showed for $t \ge 15$ and Wilson~\cite{Wi} for every $t$ that the optimal threshold is $n=(k-t+1)(t+1)$. Finally, Ahlswede and Khachatrian~\cite{AK} determined the maximum families for all values of $n$ and proved the following:

\begin{theorem}(\cite{AK})
\label{thm:AKmain}

Let $t\leq k \leq n$ and $\mathcal{A}_{n,k,t,i}=\{A : A \subseteq [n], |A|=k, |A \cap [t+2i]|\geq t+i \}$.

If ${\mathcal F} \subseteq \binom{[n]}{k}$ is t-intersecting, then we have
$$|{\mathcal F}|\leq \max_i|\mathcal{A}_{n,k,t,i}|=:AK(n,k,t).$$
\end{theorem}

Note also that this result is also a solution to an isodiametric problem.

\vspace{5mm}

\subsection{Multiset context, definitions, notation}

We think of $k$-multisets as choosing $k$ elements of $[n]$ with repetition and without ordering 
(so there are $\binom{n+k-1}{k}$ $k$-multisets). 
Let $m(i,F)$ show how many times we chose the element $i$. 
Let us define two further equivalent representations. 

\begin{defn}
\

$\bullet$ A \textit{multiset} $F$ of $[n]$ is a sequence $(m(1,F),m(2,F),...,m(n,F)) \in \mathbb{R}^n$ of $n$ natural numbers.

We call $m(i,F)$ the \textit{multiplicity} of $i$ in $F$, $\sum_{i=1}^{n} m(i,F)$ the \textit{cardinality} of $F$. We denote the cardinality

of $F$ by $|F|$ and we say that $F$ is a \textit{k-multiset} if $|F|=k$.

\vspace{2mm}

$\bullet$ The \textit{intersection} of two multisets $G$ and $F$ is a multiset defined as

$(\min\{m(1,F),m(1,G)\},\min\{m(2,F),m(2,G)\},...,\min\{m(n,F),m(n,G)\})$.

\vspace{2mm}

$\bullet$ We will use the notation $\mathcal{M}(n,k,t):=\{\mathcal{F} : \mathcal{F}$ is $t$-intersecting set of $k$-multisets of $[n]\}$.
\end{defn}

For example $F$=(3,1,2,0,0), $G$=(2,2,0,1,1) and $F\cap G$=(2,1,0,0,0) 
with this notation, if $F=\{a,a,a,b,c,c\}$ and $G=\{a,a,b,b,d,e\}$ are $6$-multisets of a five element set (and $F\cap G=\{a,a,b\}$).

It easily follows by the definition that $k$-multisets lie on the intersection of $[k]^n$ and the hyperplane 
$\sum_{i=1}^n x_i=k$. 
Concerning cardinality of $F\cap G$ we have 
$$|F\cap G|=\sum_{i=0}^n(\min\{m(i,F),m(i,G)\}=
\sum_{i=0}^n\frac{1}{2}(m(i,F)+m(i,G)-|m(i,F)-m(i,G)|)=k-\frac{1}{2}d_{\ell_1}(F,G),$$
(where $d_{\ell_1}$ denotes the $\ell_1$ distance). 
This means that a lower bound on the cardinality of the intersection of two elements gives an upper bound on their $\ell_1$ distance. So again, an upper bound on the cardinality of a $t$-intersecting family of $k$-multisets gives a result for an isodiametric problem (on the intersection of a hyperplane and a cube).

We give a third representation of multisets, which we will use in the proofs of our theorems. 
Let $n$ and $\ell$ be positive integers and let $M(n,\ell):=\{(i,j) : 1 \leq i \leq n, 1\leq j \leq \ell\}$ be a 
 $\ell\times n$ rectangle with $\ell$ rows and $n$ columns.
We call $A \subseteq M(n,\ell)$ a $k$-\emph{multiset} if the cardinality of $A$ is $k$ and $(i,j)\in A$ implies $(i,j')\in A$ for all $j' \leq j$. Certainly $m(i,F)=\max \{s : (i,s) \in F \}$ gives the equivalence with our original definition.

\subsection{The history of $t$-intersecting $k$-multisets}

\vspace{5mm}

Brockman and Kay stated the following conjecture:

\begin{conjecture}(\cite{BK}, Conjecture 5.2.)
\label{conj:BK}

There is $n_0(k,t)$ such that if $n \geq n_0(k,t)$ and ${\mathcal F} \in \mathcal{M}(n,k,t)$, then

$$|{\mathcal F}| \leq \binom{n+k-t-1}{k-t}.$$

Furthermore, equality is achieved if and only if each member of ${\mathcal F}$ contains a fixed $t$-multiset of $M(n,k)$.
\end{conjecture}

Meagher and Purdy  partly answered this problem. 

\begin{theorem}(\cite{MP})
If $n \geq k+1$ and ${\mathcal F} \in \mathcal{M}(n,k,1)$, then

$$|{\mathcal F}|\leq \binom{n+k-2}{k-1}.$$

If $n>k+1$, then equality holds if and only if all members of ${\mathcal F}$ contain a fixed element of $M(n,k)$.
\end{theorem}

They also gave a possible candidate for the threshold $n_0(k,t)$.

\begin{conjecture}(\cite{MP}, Conjecture 4.1.)
\label{conj:MP}

Let $k, n$ and $t$ be positive integers with $t\leq k$, $t(k-t)+2\leq n$
and ${\mathcal F} \in \mathcal{M}(n,k,t)$, then

$$|{\mathcal F}| \leq \binom{n+k-t-1}{k-t}.$$

Moreover, if $n > t(k - t) + 2$, then equality holds if and only if all members of ${\mathcal F}$ contain a fixed $t$-multiset of $M(n,k)$.
\end{conjecture}

Note that if $n < t(k - t) + 2$, then the family consisting of all multisetsof $M(n,k)$  which contain a fixed $t$-multiset of $M(n,k)$ still has cardinality $\binom{n+k-t-1}{k-t}$, but cannot be the largest. Indeed, if we fix a $(t+2)$-multiset $T$ and consider the family of the multisets $F$ with $|F \cap T|\ge t+1$, we get a larger family.

\subsection{The main result: extremal families have kernels}

\

A multiset $T$ is called a \textit{$t$-kernel} of the multiset family ${\mathcal F}$ if $|F_1 \cap F_2 \cap T|\geq t$ holds for all $F_1,F_2 \in {\mathcal F}$. Obviously such a family ${\mathcal F}$ is $t$-intersecting.
\cjref{BK} claims that an extremal ${\mathcal F} \in {\mathcal M}(n,k,t)$ has a $t$ element kernel, whenever $n$ is large. We will show that the general situation is more complex and determine the size of the maximal $t$-intersecting families for all $n \ge 2k-t$.

The main idea of our proof is the following: instead of the well-known \textit{left-compression} operation, which is a usual method in the theory of intersecting families, we define (in two different ways) an operation (we denote it by $f$ in the next theorem) on $\mathcal{M}(n,k,t)$ which can be called a kind of \textit{down-compression}.

\begin{theorem}
\label{thm:compression}
Let $1 \leq t \leq k$, $2k-t \leq n$ be arbitrary. There exists $$f: \mathcal{M}(n,k,t) \rightarrow \mathcal{M}(n,k,t)$$ satisfying the following properties:

\vspace{3mm}

$(i)$ $|\mathcal{F}|=|f(\mathcal{F})|$ for all $\mathcal{F} \in \mathcal{M}(n,k,t)$,

\vspace{1mm}

$(ii)$ $M(n,1)$ is a $t$-kernel for $f(\mathcal{F})$,

\vspace{1mm}

$(iii)$ moreover, the maximum height does not increase, 
  $\max\{ m(i,F): i\in [n], F\in \mathcal{F} \}\geq \max\{ m(i,F): i\in [n], F\in f(\mathcal{F}) \}$.
\end{theorem}

\vspace{5mm}

Using \tref{compression} we prove the following theorem which not only verifies \cjref{MP}, but also gives the maximum cardinality of $t$-intersecting families of multisets even in the case $2k-t \leq n < t(k-t)+2$.

\begin{theorem}
\label{thm:EKRmulti}

Let $1 \leq t \leq k$ and $2k-t \leq n$. If $\mathcal{F} \in \mathcal{M}(n,k,t)$ then $$|\mathcal{F}| \leq AK(n+k-1,k,t),$$
where the $AK$ function is defined in $\tref{AKmain}$.
\end{theorem}

Beside proving \tref{EKRmulti} our aim is to present the most powerful techniques of extremal hypergraph theory, namely the kernel and the shifting methods.

\subsection{Warm up before the proofs}

\

Let us remark that one can relatively easily verify \cjref{BK} for very large $n$.
Let $T$ be a $t$-multiset. For any family ${\mathcal F}$ let ${\mathcal F}_T=\{F \in {\mathcal F} : T \subseteq F\}$.


\begin{lemma}\label{trivi} Let ${\mathcal F}$ be a $t$-intersecting family of multisets and $T$ be an arbitrary $t$-multiset. Then either ${\mathcal F}_T={\mathcal F}$ or $|{\mathcal F}_T|=O_n(n^{k-t-1})$.
\end{lemma}

\begin{proof} If ${\mathcal F}_T \neq {\mathcal F}$, then there is a multiset $F\in {\mathcal F}$ which does not contain $T$, hence $|F \cap T| \le t-1$. Every member of ${\mathcal F}_T$ contains $T$, one element of $F \setminus T$, and at most $k-t-1$ further elements.
The element of $F \setminus T$ can be chosen less than $k$ ways, and the other $k-t-1$ elements have to be chosen out of the $nk$ elements of the rectangle $M(n,k)$. There are at most $k\times {nk\choose k-t-1}=O_n(n^{k-t-1})$ ways to do that.
\end{proof}

\begin{corollary}[Conjecture~\ref{conj:BK}]

There is $n_0(k,t)$ such that if $n \geq n_0(k,t)$ and ${\mathcal F} \in \mathcal{M}(n,k,t)$, then
$$
  |{\mathcal F}| \leq \binom{n+k-t-1}{k-t}.$$
Furthermore, equality is achieved if and only if each member of ${\mathcal F}$ contains a fixed $t$ elements.
\end{corollary}

\begin{proof} Let ${\mathcal F} \in \mathcal{M}(n,k,t)$ of maximum cardinality. If ${\mathcal F}_T={\mathcal F}$ for a $t$-multiset $T$, the statement follows. If not, then let us fix an $F \in {\mathcal F}$. Every member of ${\mathcal F}$ contains a $t$-multiset which is also contained in $F$, hence $\bigcup \{{\mathcal F}_T: T\subset F, |T|=t\}={\mathcal F}$. Thus $|{\mathcal F}| \le \sum_{T\subset F, |T|=t} |{\mathcal F}_T|$. By Lemma~\ref{trivi} $|{\mathcal F}_T|=O_n(n^{k-t-1})$, and there are ${k \choose t}$ members of the sum, hence $|{\mathcal F}| \le {k \choose t}O_n(n^{k-t-1}) < \binom{n+k-t-1}{k-t}$ if $n$ is large enough.
\end{proof}

\vspace{5mm}

To attack~\cjref{MP} at first we developed a straight-forward generalization of shifting. However, we could not give a threshold below $\Omega(kt \log k)$ using this method. Still we believe it is worth mentioning, as it might be useful solving other related problems.

\vspace{2mm}

For $F \subseteq M(n,\ell)$
$k$-multiset and $i < j$ let us suppose $m(j,F)-m(i,F)>0$. Let $F'$ be the result if we exchange column $j$ and column $i$, i.e., $F'=F \setminus \{(j,m(i,F)+1),\dots,(j,m(j,F))\} \cup \{(i,m(i,F)+1),\dots,(i,m(j,F))\}$. Let ${\mathcal F}\in\mathcal{M}(n,k,t)$ and $F \in {\mathcal F}$. Define

\vspace{2mm}

$c_{i,j}(F) = \left\{ \begin{array}{lll}
F^{'} & \textrm{if}$ $m(F,j)-m(F,i)>0 $ and $
F^{'} \not \in {\mathcal F}\\
F & $ otherwise. $
\end{array} \right.$

\vspace{2mm}

Let us use the following notation: $c_{i,j}({\mathcal F})=\{c_{i,j}(F) : F \in {\mathcal F}\}.$ Note that this is the same as the well-known shifting operation in case ${\mathcal F}\in \mathcal{M}(n,1,k,t)$, where ${\mathcal M}(n,\ell,k,t)$ denotes those elements ${\mathcal F}$ of ${\mathcal M}(n,k,t)$, where $\max_{i \le n} m(i,F) \le \ell$ for all $F\in {\mathcal F}$.

\vspace{2mm}

\begin{lemma}
\label{lem:leftcompressing}
$c_{i,j}({\mathcal F})\in \mathcal{M}(n,\ell,k,t)$ $(i<j)$ for ${\mathcal F} \in \mathcal{M}(n,\ell,k,t)$.
\end{lemma}

\begin{proof} Suppose, by contradiction, that there are $F_1,F_2 \in {\mathcal F}$ with $|c_{i,j}(F_{1}) \cap c_{i,j}(F_{2})| < t$.
If both or neither of $c_{i,j}(F_{1})$ and $c_{i,j}(F_{2})$ are members of $\mathcal{F}$, their intersection obviously has size at least $t$. Hence without loss of generality we can assume $c_{i,j}(F_{1})=F'_{1} \not \in {\mathcal F}$ and $c_{i,j}(F_{2})=F_{2} \in {\mathcal F}$.

Let $x$ be the cardinality of the intersection of $F_1$ and $F_2$ in the complement
of the union of the $i$'th and $j$'th column.

\vspace{2mm}
\textit{Case 1:} $m(F_{2},i) \geq m(F_{2},j)$.  We know that $m(F_{1},i)<m(F_{1},j)$,

$$x+\min\{m(F_{2},i), m(F_{1},i)\}+\min\{m(F_{2},j), m(F_{1},j)\}
\geq t \hskip0.3cm \textrm{and}$$
$$x+\min\{m(F_{2},i), m(F_{1},j)\}+\min\{m(F_{2},j), m(F_{1},i)\} < t.$$

\vspace{2mm}

If $m(F_1,j)$ is the largest of the four numbers, then
$\min\{m(F_{2},i), m(F_{1},i)\}+m(F_{2},j) >m(F_{2},i)+\min\{m(F_{2},j), m(F_{1},i)\}$, but here the left hand side is at least $m(F_{2},j)+m(F_{2},i)$, the right hand side can be smaller only if $m(F_{1},i)<m(F_{2},j)$, but in that case one can easily see that the left hand side is even smaller. If not $m(F_1,j)$, then $m(F_2,i)$ is (one of) the largest of the four numbers, and the proof goes similarly.

\textit{Case 2}: $m(F_{2},i) < m(F_{2},j)$ but $F^{'}_{2}\in {\mathcal F}$. We know that
$|F^{'}_{1} \cap F_{2}|=|F^{'}_{2} \cap F_{1}|\ge t,$ a
contradiction.
\end{proof}

\begin{remark}
It's worth mentioning that there is an even more straightforward generalization of shifting, when we just decrease the multiplicity in column $j$ by one and increase it in column $i$ by one. Let $F'=F \cup \{(i, m(i,F))+1\} \setminus \{(j,m(j,F))\}$, and

\vspace{2mm}

$c_{i,j}'(F) = \left\{ \begin{array}{lll}
F' & \textrm{if}$ $m(j,F)>m(i,F) $ and $
F' \not \in {\mathcal F}\\
F & $ otherwise. $
\end{array} \right.$

\vspace{2mm}

But this operation does not preserve the $t$-intersecting property. However, if we apply our shifting operation to a maximum $t$-intersecting family and for every pair $(i,j)$, the resulting family will be also shifted according to this second kind of shifting, meaning that applying this second operation does not change the family.
\end{remark}

After applying $c_{i,j}$ for every pair $i,j$, the resulting shifted family has several
different $t$-kernels. For example the union of two rectangles $t^{1/2} \times (2k-t) \cup t \times \frac{2k}{t^{1/2}}$ is a $t$-kernel, or the set where the members $(x,y)$ satisfy $yx\leq k, x\leq 2k-t, y\leq k, x,y \geq 0$.

But using these $t$-kernels and some algebra, we failed to break through the $O(kt \log k) \leq n$ barrier. 
That's the reason we had to develop the \textit{down compressing} techniques described in the next section.

\section{Proof of \tref{compression}}

Both methods described below have underlying geometric ideas, it is a kind of discrete, tilted version of symmetrizing a set with respect to the hyperplane $x_i=x_j$.

\subsection{First proof --- a constructive one}

\subsubsection{A lemma about interval systems}

We will consider multisets which contain almost exactly the same elements, they differ only in two columns. More precisely, we are interested in multisets whose symmetric difference is a subset of $\{(i,j): 1\le j \le k\}\cup\{(i',j): 1\le j \le k\}$ with $i\neq j$. If we are given two such multisets, we can consider the two columns together, one going from $k$ to $1$ and the other going from $1$ to $k$. This way the restriction of two multisets to these columns form a subinterval of an interval of length $2k$ (where interval means set of consecutive integers). Hence we examine families of intervals.

Let $X:=\{1,\dots, 2k\}$. Let $Y \subset X$ be an interval and $p$ be an integer with $p \le |Y|$, then we define a family of intervals $I(p,Y)$ to be all the $p$-element subintervals of $Y$. Let ${\mathcal I}$ denote the class of these families.

We consider a shifted version, where the intervals are pushed to the middle. Let $\varphi(Y):=\{k-\lfloor|Y|/2\rfloor,\dots,k+\lceil|Y|/2\rceil-1\}$ and $\varphi(I(p,Y))=I(p,\varphi(Y))$. We will show that this operation does not decrease the size of the intersection of two families in ${\mathcal I}$.
Let $d(I(p,Y), I(q, Y')) := \min \{|I\cap J| : I \in I(p,Y), J \in I(q, Y')\}$.

\vspace{5mm}

\begin{lemma}
\label{lem:intervalsystem}

$d(I(p,\varphi(Y)),I(q,\varphi(Y'))) \geq d(I(p,Y), I(q, Y'))$ for all possible $p$, $q$, $Y$ and $Y'$.
\end{lemma}

\begin{proof} Obviously the smallest intersection is the intersection of the first interval in one of the families and the last interval in the other family. As the length of the intervals are always $p$ and $q$, the only thing that matters is the difference between the starting and ending points of $\varphi(Y)$ and $\varphi(Y')$. More precisely we want to minimize the largest of $y=\max \varphi(Y)-\min\varphi(Y')$ and $y'=\max \varphi(Y')-\min\varphi(Y)$. As $\max \varphi(Y)-\min\varphi(Y)=|Y|-1$ and $\max \varphi(Y')-\min\varphi(Y')=|Y'|-1$, we know that $y+y'$ is constant, hence we get the minimum if $y$ and $y'$ is as close as possible.
One can easily see that our shifted system gives this.
\end{proof}

\subsubsection{Interval systems and families of multisets}

Now to apply the method of the previous subsection, we fix $n-2$ coordinates, i.e., we are given $i<j\le n$ and $g: ([n] \setminus \{i,j\})\rightarrow [1,k]$. Let $\mathcal{F}_g=\{F\in\mathcal{F}: m(r,F)=g(r)\, \textrm{for every $r\neq i,j$}\}$. It implies $m(i,F)+m(j,F)$ is the same number $s=s(g)$ for every member $F\in\mathcal{F}_g$.

Let us consider now the case ${\mathcal F}$ is maximal, i.e., no $k$-multiset can be added to it without violating the $t$-intersecting property. We show that it implies that the integers $m(i,F)$ are consecutive for $F\in \mathcal{F}_g$.
Let $m_i=\min \{m(i,F): F \in \mathcal{F}_g\}$ and $M_i=\max \{m(i,F): F \in \mathcal{F}_g\}$. 
We define $m_j$ and $M_j$ similarly. Let us consider a set $F\not\in \mathcal{F}$ which satisfies $m(r,F)=g(r)$ for all $r\neq i,j$ and also $m_i \le m(i,F)\le M_i$, and consequently $m_j \le m(j,F) \le M_j$. It is easy to see that $F$ can be added to $\mathcal{F}$ without violating the $t$-intersecting property (and then it belongs to $\mathcal{F}_g$).

Now we give a bijection between these type of families and interval systems. We lay down both columns, such that column $i$ starts at its top, and column $j$ start at its bottom. Then move them next to each other to form an interval. More precisely let $\Psi((i,u))=k-u+1$ and $\Psi((j,u))= k+u$. For a multiset $F$ let $\Psi(F)=\{\Psi((i,u)): (i,u) \in F\} \cup \{\Psi((j,u)): (j,u) \in F\}$ and for a family of multisets ${\mathcal F}$ let $\Psi({\mathcal F})=\{\Psi(F): F \in {\mathcal F}\}$.

We show that $\Psi({\mathcal F}_g) \in {\mathcal I}$. It is obvious that $\Psi(F)$ is an interval for any multiset $F$, and that the length of those intervals is the same number (more precisely $s$) for every member $F\in\mathcal{F}_g$. We need to show that the intervals $\Psi(F)$ (where $F \in {\mathcal F}_g$) are all the subintervals of an interval $Y$. It is enough to show that the starting points of these intervals are consecutive integers. The starting points of the intervals $\Psi(F)$ are $\Psi((i,m(i,F)))$, and it is easy to see that they are consecutive if and only if $m(i,F)$ are consecutive.

Since $\Psi$ is a bijection, an interval system also defines a family in the two columns $i$ and $j$. Let us examine what family we get after applying operation $\varphi$ from the previous section, i.e., what ${\mathcal F}'= \Psi^{-1}(\varphi(\Psi({\mathcal F}_g)))$ is. Obviously it is a family of $s$-multisets with the same cardinality as ${\mathcal F}_g$. Simple calculations show that they are the $s$-multisets with $m(i,F) \le \lfloor /m_i+M_j)/2\rfloor+1$ and $m(j,F) \le \lceil (m_i+M_j)/2\rceil -1$.

\subsubsection{The construction of $f$}

Let $\psi(\mathcal{F}_g)=\Psi^{-1}(\varphi(\Psi({\mathcal F}_g)))$, i.e., the family we get from $\mathcal{F}_g$ by keeping everything in the other $n-2$ columns, but making it balanced in the columns $i$ and $j$ in the following sense. It contains all the $k$-multisets where $m(i,F) \le \lfloor (m_i+M_j)/2\rfloor+1$, $m(j,F) \le \lceil (m_i+M_j)/2\rceil -1$ and the other coordinates are given by $g$.

Now let us recall that $i$ and $j$ are fixed. Let $G_{i,j}$ be the set of every $g: ([n] \setminus \{i,j\})\rightarrow [1,k]$, i.e., every possible way to fix the other $n-2$ coordinates. Clearly $\mathcal{F}=\cup \{\mathcal{F}_g: g \in G_{i,j} \}$ and they are all disjoint. Let $\psi_{i,j}(\mathcal{F})$ denote the result of applying the appropriate $\psi$ operation for every $g$ at the same time, i.e., $\psi_{i,j}(\mathcal{F})=\cup \{\psi(\mathcal{F}_g): g \in G_{i,j}\}$.

\begin{lemma} If $\mathcal{F}$ is $t$-intersecting, then $\psi_{i,j}(\mathcal{F})$ is $t$-intersecting.
\end{lemma}

\begin{proof} Suppose there are $F_1,F_2\in \psi_{i,j}(\mathcal{F})$ with $|F_1\cap F_2| <t$. Let $F_1\in \psi_{i,j}(\mathcal{F}_{g_1})$ and  $F_2\in \psi_{i,j}(\mathcal{F}_{g_2})$. Let $\Psi(\mathcal{F}_{g_1})=I(p_1,Y_1)$ and
$\Psi(\mathcal{F}_{g_2})=I(p_2,Y_2)$. Then $\Psi(\psi_{i,j}(\mathcal{F}_{g_1}))=\varphi(I(p_1,Y_1))$ and $\Psi(\psi_{i,j}(\mathcal{F}_{g_2}))=\varphi(I(p_2,Y_2))$. It is important to see that $\Psi$ is defined on the elements of $M(n,\ell)$ such a way that the size of the intersection is the same after applying $\Psi$.

By Lemma~\ref{lem:intervalsystem} $d(I(p_1,\varphi(Y_1)),I(p_2,\varphi(Y_2))) \geq d(I(p_1,Y_1), I(p_2, Y_2))$, which means there is a member of $\mathcal{F}_{g_1}$ and a member of $\mathcal{F}_{g_2}$ such that their intersection has size at most the size of the smallest intersection between members of $\psi_{i,j}(\mathcal{F}_{g_1})$, which is less than $t$, a contradiction.

Note that it does not matter if $g_1$ is equal to $g_2$.
\end{proof}

\begin{lemma}\label{csokken}

If $\psi_{i,j}({\mathcal F})\neq {\mathcal F}$ then

\begin{equation}
\label{eq:3}
\sum_{F'\in \psi_{i,j}({\mathcal F})}\Big[|{\mathcal F}|nk^2\sum_{i \in [n]}(m(i,F'))^2+\sum_{i \in [n]}i(m(i,F'))\Big]< \sum_{F\in {\mathcal F}}\Big[|{\mathcal F}|nk^2\sum_{i \in [n]}(m(i,F))^2+\sum_{i \in [n]}i(m(i,F))\Big]
\end{equation}
\end{lemma}

\begin{proof}
Trivial by the symmetrization.
\end{proof}

Now we are ready to define $f(\mathcal{F})$. If there is a pair $(i,j)$ such that $\psi_{i,j}(\mathcal{F})\neq \mathcal{F}$, let us replace $\mathcal{F}$ by $\psi_{i,j}(\mathcal{F})$, and repeat this step. Lemma~\ref{csokken} implies that it can be done only finitely many times, after that we arrive to a family ${\mathcal F}'$ such that $\psi_{i,j}(\mathcal{F'})=\mathcal{F'}$ for every pair $(i,j)$. This family is denoted by $f(\mathcal{F})$.

We would like to prove that $f$ satisfies \tref{compression} $(ii)$.
This step is the only point we use that $n\geq 2k-t$. 

\begin{lemma}

$|F_1\cap F_2 \cap M(n,1)| \geq t$ for all $F_1, F_2 \in f(\mathcal{F})$.
\end{lemma}

\begin{proof}
We argue by contradiction. Let us choose $F_1$ and $F_2$ such a way that $|F_1\cap F_2 \cap M(n,1)|$ is the smallest (definitely less than $t$), and among those $|F_1 \cap F_2|$ is the smallest (definitely at least $t$).  Then there is a coordinate where both $F_1$ and $F_2$ have at least $2$, and this implies there is an other coordinate, where both have $0$, as $2k-t \leq n$. More precisely, there is an $i \leq n $ with $2 \leq \min\{m(i,F_1),m(i,F_2)\}$ and a $j \leq n$ with $m(j,F_1)=m(j,F_2)=0$. Let $F'_1$ be defined the following way: $m(j,F'_1)=1, m(i,F'_1)=m(i,F_1)-1$ and $m(s,F'_1)=m(s,F_1)$ for $s \leq n$, $s\neq i,j$. One can easily see that $F_1' \in \psi_{i,j}(f(\mathcal{F}))=f(\mathcal{F})$. However, $|F_1' \cap F_2| < |F_1 \cap F_2|$ and $|F_1'\cap F_2 \cap M(n,1)| = |F_1\cap F_2 \cap M(n,1)|$, a contradiction.
\end{proof}

To finish the proof of \tref{compression} we have to deal with the case ${\mathcal F}$ is not maximal (even though it is not needed in order to prove \tref{EKRmulti}). For sake of brevity here we just give a sketch.

Note that $\Psi$ can be similarly defined in this case. The main difference is that the resulting family of intervals is not in ${\mathcal I}$, as it does not contain all the subintervals of an interval. Also note that $\varphi(I(p,Y))$ is determined by the number and length of the intervals in $I(p,Y)$. Using this we can extend the definition of $\varphi$ to any family of intervals. This way we can define $\psi_{i,j}$ as well. What happens is that besides being more balanced in the columns $i$ and $j$, the multisets in ${\mathcal F}_g$ are also pushed closer to each other. Hence one can easily see that the intersections cannot be smaller in this case, which finishes the proof.


\subsection{Second proof --- a less constructive one}

\begin{proof}
[Proof of \tref{compression}:]
\

\vspace{3mm}

For $F \in {\mathcal F} \in \mathcal{M}(n,k,t)$, $i \leq n$, $s \leq m(i,F)$ and $j \leq n$ let

$$F'=F\setminus (\cup_{s \leq t \leq m(i,F)} (i,t)) \bigcup (\cup_{1 \leq l \leq m(i,F)-s+1}(j,l)).$$

Using this notation we define another shifting operation.

\vspace{3mm}

\begin{defn}

$S(i,s)(j,1)(F) := \left\{ \begin{array}{ll}
F' & \textrm{if}\,  (j,1) \not \in F \,\textrm{and}\, F' \not \in {\mathcal F}\\
\\

F &  \textrm{otherwise.}
\end{array} \right.$

\vspace{2mm}

For ${\mathcal F} \in \mathcal{M}(n,k,t)$ let $S(i,s)(j,1)({\mathcal F})= \{S(i,s)(j,1)(F) : F \in {\mathcal F} \}$.
\end{defn}

\vspace{2mm}

For ${\mathcal F} \in \mathcal{M}(n,k,t)$ let ${\mathcal K}({\mathcal F})$ be the set of \textit{t-kernels} of ${\mathcal F}$ which contain $M(n,1)$ and are multisets. For $T \in {\mathcal K}({\mathcal F})$ let $T_{>1}:= T \setminus M(n,1)$.
We would like to define an operation on $\mathcal{M}(n,k,t)$ which decreases $\min\{|T_{>1}| : T \in {\mathcal K}({\mathcal F})\}$ for any ${\mathcal F} \in \mathcal{M}(n,\ell,k,t)$ in case it is positive.

Let us apply $S(i,m(i,T))(1,1)$ to $\mathcal{F}$, then $S(i,m(i,T))(2,1)$ to the resulting family, and so on. Let ${\mathcal F}'$ be the resulting family after applying $S(i,m(i,T))(n,1)$, i.e

$${\mathcal F}'=S(i,m(i,T))(n,1)[....[S(i,m(i,T))(2,1)[S(i,m(i,T))(1,1)({\mathcal F})]]...].$$

\vspace{4mm}

\begin{lemma}
\label{lem:comp}

Let ${\mathcal F} \in \mathcal{M}(n,k,t)$, $T \in {\mathcal K}({\mathcal F})$ satisfying
$|T_{>1}| > 0$ and let $1 \leq i \leq n$, $2\leq m(i,T)$.
Then:

\vspace{2mm}

$(i)$ ${\mathcal F}' \in \mathcal{M}(n,k,t)$ and $|{\mathcal F}|=|{\mathcal F}'|$;

\vspace{2mm}

$(ii)$ $(T\setminus (i,m(i,T))) \in {\mathcal K}({\mathcal F}')$.
\end{lemma}
\begin{proof}

\

\vspace{2mm}

$\bullet$ proof of $(i)$:

\vspace{2mm}

The facts that $S(i,m(i,T))(n,1)[....[S(i,m(i,T))(2,1)[S(i,m(i,T))(1,1)(F)]]...] \subseteq M(n,k)$ with cardinality $k$ for any $F \in {\mathcal F}$ and that $|{\mathcal F}'|=|{\mathcal F}|$, are trivial.

\vspace{2mm}

\begin{claim}
${\mathcal F}'$ is \textit{t-intersecting}.
\end{claim}

\begin{proof}

It is enough to prove that $S(i,m(i,T))(1,1)({\mathcal F})$ is $t$-intersecting and that $T$ is a $t$-kernel for the new family, 
i.e., $T \in {\mathcal K}(S(i,m(i,T))(1,1)({\mathcal F}))$, since repeatedly applying this fact we will get the claim.

\vspace{3mm}

Choose two arbitrary members. As usual, it is easy to handle the cases when both or neither is a member of the original family ${\mathcal F}$. 
Hence we can assume without loss of generality that we are given $F,G\in {\mathcal F}$ with $S(i,m(i,T))(1,1)(F) \ne F$ but $S(i,m(i,T))(1,1)(G) = G$. Then $(1,1) \not \in F$.

Now if $(1,1) \in G$ then the intersection (of the two set and the kernel $T$) increases by one and decreases by at most one, we are done. Otherwise $S(i,m(i,T))(1,1)(G) \in {\mathcal F}$, and then we have $t \le |S(i,m(i,T))(1,1)(G) \cap F \cap T| \le |S(i,m(i,T))(1,1)(F) \cap G \cap T|.$
\end{proof}

\vspace{2mm}

We are done with the proof of $(i)$ of \lref{comp}.

\vspace{2mm}

$\bullet$ proof of $(ii)$:

\vspace{2mm}
We choose $F,G \in {\mathcal F}$ and prove that $|F' \cap G' \cap (T \setminus (i,m(i,T)))|\ge t$:

\vspace{2mm}

\textit{Case 1}: if $$S(i,m(i,T))(n,1)[....[S(i,m(i,T))(2,1)[S(i,m(i,T))(1,1)(F)]]...]= F' \ne F \hskip0.2truecm or$$ $$S(i,m(i,T))(n,1)[....[S(i,m(i,T))(2,1)[S(i,m(i,T))(1,1)(G)]]...]=G' \ne G,$$

since $T$ is a \textit{t-kernel} for ${\mathcal F}$, however $(i,m(i,T)) \not \in F' \cap G'$, we are done similarly as in the previous claim.

\vspace{4mm}

\textit{Case 2}: if $$S(i,m(i,T))(n,1)[....[S(i,m(i,T))(2,1)[S(i,m(i,T))(1,1)(F)]]...]=F\hskip0.2truecm and$$  $$S(i,m(i,T))(n,1)[....[S(i,m(i,T))(2,1)[S(i,m(i,T))(1,1)(G)]]...]=G,$$

then

\vspace{2mm}

a) if $(i,m(i,T)) \not \in F \cap G$, we are done,

b) if $(i,m(i,T)) \in F \cap G$ then since $2 \le m(i,T)$ and $2k-t \le n$, there is $j \le n$ with $(1,j) \not \in F \cup G$.

\vspace{3mm}

Then as we are in the \textit{Case 2},
$S(i,m(i,T))(j,1)(F) \in {\mathcal F}$, so $$t \le |S(i,m(i,T))(j,1)(F) \cap G \cap T| = |F \cap G \cap (T \setminus (i,m(i,T)))|.$$

\vspace{5mm}

We are done with the proof of \lref{comp}.
\end{proof}

We are done with the proof of \tref{compression}.
\end{proof}

\section{Proof of \tref{EKRmulti}}

Let ${\mathcal G}_s = \{F \cap M(n,1): F \in f({\mathcal F}), \, |F \cap M(n,1)|=s \}$. Let us consider $G \in {\mathcal G}_s$ and examine the number of multisets $F\in {\mathcal F}$ with $G=F\cap M(n,1)$. Obviously $k-s$ further elements belong to $F$, and they are in the same $s$ columns, they can be chosen at most $\binom{s+(k-s-1)}{k-s}$ ways. Then we know that

$$|{\mathcal F}|=|f({\mathcal F})|\le \sum_{s=t}^{k}|{\mathcal G}_s|\binom{s+(k-s-1)}{k-s}=\sum_{s=t}^{k}|{\mathcal G}_s|\binom{k-1}{k-s}.$$

Now consider a family ${\mathcal F}'$ of sets on an underlying set of size $n+k-1$. Let it be the same on the first $n$-elements as $f({\mathcal F})$ in $M(n,1)$, and extend every $s$-element set there with all the $(k-s)$-element subsets of the remaining $k-1$ elements of the underlying set. It can happen $\binom{k-1}{k-s}$ ways, thus the cardinality of this family is the right hand side of the above inequality.

Note that ${\mathcal F}'$ is $t$-intersecting, hence its cardinality is at most $AK(n+k-1,k,t)$, which completes the proof of \tref{EKRmulti}.

\section{Concluding remarks}

Note that the bound given in \tref{EKRmulti} is sharp. Using a family ${\mathcal A}_{n+k-1,k,t,i}$ we can define an optimal $t$-intersecting family of $k$-multisets in $M(n,k)$. However, we do not know any nontrivial bounds in case of $n<2k-t$.

After repeated down-shifting we get a following structure theorem.
Let $e_i\in \mathbb{R}^n$ denote the standard unit vector with $1$ in its $i$'th coordinate.

\begin{lemma}[Stable extremal families] 
There exists a family ${\mathcal F} \in \mathcal{M}(n,\ell, k,t)$ of maximum cardinality satisfying 
 the following two properties: \\ --- $\forall i\neq j$ and $F\in {{\mathcal F}}, 
  \enskip m(i,F)+1 < m(j,F)$  imply that  \enskip $(F-e_j +e_i) \in {\mathcal F}$, too, and \\
--- the same holds if $i< j$ and $m(i,F)+1 \leq m(j,F)$.
  \end{lemma}
Knowing the structure of ${\mathcal F}$ might help to determine $\max|{\mathcal F}|$ for all $n$.


The original Erd\H{o}s-Ko-Rado theorem (\tref{EKRorig}) concerns the maximum size of an independent set in the Kneser graph. A powerful method to estimate the size of an independent set was developed by Lov\'asz. Indeed, Wilson~\cite{Wi} extended the Erd\H{o}s-Ko-Rado theorem by determining the Shannon capacity of the generalized Kneser graph. It would be interesting if his ideas were usable for the multiset case, too.


\begin{thebibliography}{alpha}

\bibitem{AK} R. Ahlswede and L.H. Khachatrian: {\em The complete intersection theorem for systems of finite sets}, European J. Combin., 18(2) (1997), 125--136.
\bibitem{AK2} R. Ahlswede and L.H. Khachatrian: {\em The Diametric Theorem in Hamming Spaces -- Optimal Anticodes}, Advances in Applied Mathematics 20 (1998), 429--449.
\bibitem{BL} B. Bollob\'as and I. Leader: {\em Maximal sets of given diameter in the grid and the torus}, Discrete Mathematics 122 (1993), 15--35.
\bibitem{BK} G. Brockman and B. Kay: {\em Elementary Techniques for Erd\H{o}s-Ko-Rado-like Theorems}, http://arxiv.org/pdf/0808.0774v2.pdf.
\bibitem{DK} D.Z. Du and D.J. Kleitman: {\em Diameter and radius in the Manhattan metric},
Discrete Comput. Geom. 5 (1990), no. 4, 351--356.
\bibitem{EKR} P. Erd\H{o}s, C. Ko and R. Rado: {\em Intersection theorems for systems of finite sets}, Quart. J. Math. Oxford Ser. (2), 12 (1961), 313--320.
\bibitem{Fr} P.  Frankl: {\em The shifting technique in extremal set theory}, in: Surveys in Combinatorics, Lond. Math. Soc. Lect. Note Ser., 123 (1987), 81--110.
\bibitem{HR} A. Hajnal and B. Rothschild: {\em A generalization of the Erd\H{o}s-Ko-Rado theorem on finite set systems}, 
Journal of Combinatorial Theory, Ser. A 15 (1973), no.3, 359--362.
\bibitem{Ka} G.O.H. Katona: {\em Intersection theorems for systems of finite sets}, Acta Math. Hungar. 15 (1964), 329--337.
\bibitem{Kl} D.J. Kleitman: {\em On a combinatorial conjecture of Erd\H{o}s}, J. Combin. Theory 1 (1966), 209--214.
\bibitem{MP} K. Meagher and A. Purdy: {\em  An Erd\H{o}s-Ko-Rado theorem for multisets}, Electron. J. Combin. 18 (2011), no.1, Paper 220.
\bibitem{Me} M.S. Mel'nikov: {\em Dependence of volume and diameter of sets in an n-dimensional Banach space}, Uspehi Mat. Nauk 18 (1963) no. 4 (112), 165--170.
\bibitem{Wi} R.M. Wilson: {\em The exact bound on the Erd\H{o}s-Ko-Rado theorem}, Combinatorica, 4 (1984), 247--257.

\end{thebibliography}
\end{document}